\documentclass[a4paper, 11pt, leqno,bibliography=totoc]{scrartcl}

\usepackage[english]{babel}			
\usepackage[utf8]{inputenc} 		
\usepackage[T1]{fontenc}		
\usepackage{scrlayer-scrpage}						
\usepackage{lastpage}
\usepackage{hyperref}						
\usepackage[dvipdf, pdftex]{graphicx}		
\usepackage{wrapfig}						
\usepackage{picinpar}						
\usepackage[margin=2ex,font=small,labelfont=bf]{caption}		
\usepackage{mathtools}
\usepackage{amsthm}
\usepackage{amssymb}
\usepackage{setspace}						
\usepackage{graphicx}
\usepackage{stackengine}
\usepackage[noadjust]{cite}
\usepackage{authblk}
\usepackage{abstract}
\usepackage[all,defaultlines=3]{nowidow}
\usepackage{enumitem}
\usepackage[normalem, normalbf]{ulem}
\mathtoolsset{showonlyrefs}

\stackMath

\hypersetup{
 pdftoolbar=true, pdfmenubar=true, pdftitle={Asymptotics for axisymmetric Naiver-Stokes}, pdfauthor={Dominik Winkler}, linkcolor=black, citecolor=black, urlcolor=red, colorlinks=true
}

\theoremstyle{plain}
\newtheorem{theorem}{Theorem}[section]
\newtheorem*{theorem*}{Theorem}
\newtheorem{lemma}[theorem]{Lemma}
\newtheorem*{lemma*}{Lemma}

\newtheorem*{korollar*}{Korollar}
\newtheorem{proposition}[theorem]{Proposition}
\newtheorem*{proposition*}{Proposition}

\newtheorem*{satz*}{Satz}
\newtheorem{corollary}[theorem]{Corollary}
\newtheorem*{corolarry*}{Corollary}

\newtheoremstyle{named}{\topsep}{\topsep}{\itshape}{0pt}{\bfseries}{}{5pt plus 1pt minus 1pt}{\thmname{#1}\thmnumber{ #2.}\thmnote{\;(#3)}}
\theoremstyle{named}
\newtheorem*{namedtheorem*}{Theorem}
\newtheorem{namedtheorem}[theorem]{Theorem}
\newtheorem{namedlemma}[theorem]{Lemma}
\newtheorem*{namedlemma*}{Lemma}

\newtheorem*{namedkorollar*}{Korollar}
\newtheorem{namedproposition}[theorem]{Proposition}
\newtheorem*{namedproposition*}{Proposition}
\newtheorem{namedcorollary}[theorem]{Corollary}
\newtheorem*{namedcorolarry*}{Corollary}

\newtheorem{theoremA}{Theorem}

\theoremstyle{definition}

\newtheorem*{definition*}{Definition}

\newtheorem*{beispiel*}{Beispiel}

\newtheorem*{bemerkung*}{Bemerkung}
\newtheorem*{erinnerung*}{Erinnerung}

\newtheorem*{remark*}{Remark}

\newtheorem*{assumption*}{Assumption}

\newtheorem*{notation*}{Notation}

\newcommand{\laplace}{\Delta}
\addto\captionsenglish{\renewcommand{\proofname}[2][]{Proof{#1}{#2}. }}
\renewenvironment{proof}[1][]{{\noindent\textbf{\proofname{#1}}}}{\qed}

\newcommand{\I}{\mathcal{I}}

\renewcommand{\L}{\mathcal{L}}

\newcommand{\R}{\mathbb{R}} 	
\newcommand{\N}{\mathbb{N}} 	

\newcommand{\grad}{\nabla}
\renewcommand{\div}{\grad\cdot}
\newcommand{\mel}{\MoveEqLeft}
\renewcommand{\la}{\langle}
\renewcommand{\ra}{\rangle}

\newcommand{\ltwomunorm}[1]{\left\| #1 \right\|_{L^2\left(\mu\right)}}
\newcommand{\ltwoscalarproduct}[2]{\langle #1, #2\rangle_\mu}
\newcommand{\absolutevalue}[1]{\left|#1\right|}
\newcommand{\myvector}[1]{\ensuremath{\begin{pmatrix}#1\end{pmatrix}}}

\renewcommand{\H}{\mathbb{H}}

\begin{document}

\title{Fine large-time asymptotics for the axisymmetric Navier--Stokes equations}

\author{Christian Seis\footnote{Institut f\"ur Analysis und Numerik,  Universit\"at M\"unster. Email: seis@uni-muenster.de} }
			\author{Dominik Winkler\footnote{Institut f\"ur Analysis und Numerik,  Universit\"at M\"unster. Email: dominik.winkler@uni-muenster.de}}
			
	\affil{}
	\date{\today}

	\maketitle
     
     \begin{abstract}
          We examine the large-time behavior of axisymmetric solutions without swirl of the Navier--Stokes equation in $\R^3$. We construct higher-order asymptotic expansions for the corresponding vorticity.
          The appeal of this work lies in the simplicity of the applied techniques: Our approach is completely based on standard $L^2$-based entropy methods.
     \end{abstract}

\section{Introduction and results}

In the present work, the large-time behavior of axisymmetric solutions to the three-dimensional Navier--Stokes equations is investigated. Our aim is to describe intermediate asymptotics of any solution near equilibrium. 
Even though the assumed axisymmetry reduces the dimension  of the problem and thus the level of its difficulty, some rewarding examples, e.g.\ vortex rings, are contained in this class of incompressible flows. Moreover, opposed to solutions in the planar two-dimensional setting, axisymmetric solutions are exposed to vortex stretching and carry thus certain genuinely three-dimensional features.

We start by recalling that the full three-dimensional  Navier--Stokes equations are the following system of partial differential equations
     \begin{equation}\label{navierstokesequations}
     \begin{aligned}
          \partial_t u -\Delta_x u + u\cdot \nabla_x  u +\nabla_x p &=0 \quad \text{ in } \R^3\times \left[0,\infty\right),\\
          \nabla_x\cdot  u &= 0\quad \text{ in } \R^3\times  \left[0,\infty\right),
     \end{aligned}
     \end{equation}
where $u=u\left(x,t\right) \in \R^3$ denotes the velocity  of the fluid and $p=p\left(x,t\right) \in \R$ is the pressure. For notational simplicity, we have rescaled the equation in such a way  that the kinetic viscosity and the density of the fluid both are normalized to unity.

To describe axisymmetric solutions, we introduce the   usual cylindrical coordinates $\left(r,z,\theta\right)$  in $\R^3$ defined via $x = \left(r\cos(\theta),r\sin(\theta),z\right)$ with unit basis vectors given by  
     \begin{align}
          e_r = \myvector{\cos(\theta)\\\sin(\theta)\\0}, \quad e_\theta = \myvector{-\sin(\theta)\\\cos(\theta)\\0},\quad e_z=\myvector{0\\0\\z}.
     \end{align}
Fluid configurations are called \emph{axisymmetric} if they     are   invariant under rotations about the vertical axis, that is, the axisymmetric fluid velocity
is independent of the angle $\theta$ so that $u = u(r,z,t)$. We furthermore assume that the \emph{swirl} component $u\cdot e_\theta$ of the velocity field vanishes. This second assumption has the effect, that the velocity field is also invariant under reflections by any plane containing the $z$-axis. The velocity field now reads as
\begin{align}\label{axisymmetricvelocityfield}
     u(x,t) = u_r\left(r,z,t\right)e_r + u_z\left(r,z,t\right)e_z.
\end{align}
To simplify the language use in the following, we will simply speak of the \emph{axisymmetric} setting when referring to \eqref{axisymmetricvelocityfield}.

Writing the vorticity in cylindrical variables, that is, $\nabla_x \times u = \left(r^{-1}\partial_\theta u_z -\partial_zu_\theta\right) e_r +\left(\partial_zu_r-\partial_ru_z\right)e_\theta + r^{-1}\left(\partial_r\left(r u_\theta\right) - \partial_\theta u_r\right) e_z$, we see that the hypothesis in  \eqref{axisymmetricvelocityfield} implies that the vorticity is determined by a scalar quantity,
     \begin{align}
          \nabla_x\times u = \omega \,  e_\theta, \quad \text{where } \omega  = \partial_zu_r - \partial_ru_z
     \end{align}
--- very similarly to the    planar two-dimensional case.
The velocity field in turn can be recovered from the scalar vorticity by solving the linear elliptic system
     \begin{equation}\label{relationvoritcityvelocity}
                \partial_z u_r -\partial_ru_z=\omega,\quad 
               \partial_r u_r +\frac{1}{r}u_r +\partial_z u_z =0
     \end{equation}
in the meridional half space $\mathbb{H} = \left\{(r,z) : r\in (0,\infty), z\in \R\right\}$, equipped with the free-slip boundary conditions $u_r = \partial_zu_r =0$ at $\left\{r=0\right\}$. Notice that the second identity in \eqref{relationvoritcityvelocity} is nothing but the incompressibility condition restated in cylindrical coordinates.

The time evolution of $\omega$ is described by the following scalar equation
	\begin{align}\label{vorticityequation}
          \partial_t \omega + u \cdot \nabla \omega - \frac{u_r}{r}\omega = \Delta\omega +\frac{1}{r}\partial_r\omega- \frac{\omega}{r^2} \quad \text{ in } \mathbb{H}\times \left[0,\infty\right) 
     \end{align}
with the homogeneous Dirichlet boundary condition $\omega =0$ at $\left\{r=0\right\}$. The second term on the left-hand side describes the advection by the flow, while the third term on the left-hand side accounts for possible vortex stretching. The occurrence of this term is different from the planar setting, where vortex stretching effects are absent. The derivative terms on   the right-hand side of \eqref{vorticityequation} come simply from the Laplacian in $\R^3$ written in cylindrical variables and the remaining expression results from differentiating the azimuthal basis vector $e_{\theta}$. Here  and in the following, we interpret the standard non-indexed differentiation symbols with respect to cylindrical coordinates, i.e., $\grad  = (\partial_r,\partial_z)^T$ and accordingly $\laplace  = \div \grad =\partial_r^2+\partial_z^2$.

The theory for the vorticity stream formulation \eqref{relationvoritcityvelocity}, \eqref{vorticityequation} of the axisymmetric Navier--Stokes equations parallels in many aspects the one in the planar two-dimensional setting, see, for instance, \cite{UkhovskiiYudovich1968,Ladyzenskaya1968,BenArtzi1994,Abidi2008,FengSverak15,GallaySverak2015}  regarding regularity and well-posedness, \cite{UkhovskiiYudovich1968,JiuXin04,HmidiZerguine09,Wu10,
JiuWuYang15,Gallay2019,Abe20,NobiliSeis22,Butta2022} for studies of the inviscid limit and \cite{GallaySverak2015,Vila2022,Liu2023} for questions regarding the  large-time behavior. For us most notable are some of the results by 
Gallay and \v{S}ver\'ak and their improvements by Vila, which we recall in the following.

\begin{theoremA}[\cite{GallaySverak2015,Vila2022}]
\label{TA}
The large-time behavior of any solution $\omega\in C^0( (0,\infty ),L^1(\mathbb{H})\cap L^\infty(\mathbb{H}))$ with $r^2 w\in L^{\infty} ( (0,\infty ),L^1(\mathbb{H} ))$ is given by
     \begin{align}\label{convergencevorticity}
          t^2\omega\left(\hat r\sqrt{t},\hat z\sqrt{t},t\right) \rightarrow \frac{\mathcal{I}(\omega_0)}{16\sqrt{\pi}} \hat re^{-\frac{\hat r^2+\hat z^2}{4}} \quad \text{ as } t \to \infty,
     \end{align}
     where $\mathcal{I}(\omega_0)$ is the fluid impulse,
          \begin{align}
          \mathcal{I}(\omega_0) \coloneqq \int r^2\omega_0(r,z)drdz < \infty,
     \end{align}
     and  the convergence holds in $L^p\left(\mathbb{H}\right)$ for all $p\in \left[1,\infty\right]$.
\end{theoremA}
We remark that we will distinguish between the Lebesgue spaces $L^p(\H)$ and $L^p(\R^3)$. Here, the first ones are defined with respect to the two-dimensional measure $drdz$ on the meridional half-plane $\mathbb{H}$ while the second one is equipped with     the full three-dimensional measure $dx$. Notice that when acting on axisymmetric functions or sets, it holds that $dx = 2\pi rdrdz$.

As expected, the fluid impulse is preserved during the evolution,
\begin{equation}
\label{111}
\mathcal{I}(\omega(t)) = \mathcal{I}(\omega_0)\,\quad \mbox{for any }t,
\end{equation}
and we will occasionally write  in the following simply  $\mathcal{I}$ instead of $\mathcal{I}(\omega_0)$.
In fact, Gallay and \v{S}ver\'ak show  well-posedness for the Navier--Stokes equation \eqref{vorticityequation} with integrable  initial data, and they establish the large-time behavior \eqref{convergencevorticity} for non-negative vorticities.  By enhancing the semi-group arguments of \cite{GallaySverak2015},
Vila then notes that the sign condition is not necessary for establishing these asymptotics. The main observation in \cite{Vila2022} is that  the quantity $\|r^2\omega\|_{L^1\left(\mathbb{H}\right)}$ remains uniformly bounded for all times, which replaces the conservation of the fluid impulse in the Gallay--\v{S}ver\'ak paper.

It is interesting to notice that the profile which prescribes the large-time behavior of the solution $\omega$ is the self-similar solution $\zeta_*$ of the \emph{linearized} equation 
     \begin{align}\label{linearizedvorticityequation}
          \partial_t\zeta =\Delta \zeta + \frac{1}{r}\partial_r\zeta -\frac{1}{r^2}\zeta, 
     \end{align}
     that is,
     \[
     \zeta_*(t,r,z) = \frac1{t^2} \rho_*\left(\frac{r}{\sqrt{t}},\frac{z}{\sqrt{t}}\right),\quad \rho_*(\hat r, \hat z) = \frac1{16\sqrt{\pi}} \hat r e^{-\frac{\hat r^2+\hat z^2}4},
     \]
     where the prefactor is chosen in such a way that it normalizes the impulse, $\I(\zeta_*) = \I(\rho_*)=1$. 
    
In a certain sense, this feature is  caused by the Dirichlet boundary conditions valid for $\omega$ at the boundary of the half-plane $\H$. Indeed, these boundary conditions enforce the decay of the (scale invariant) $L^1$ norm, and more generally, 
     \begin{equation}
     \label{1}
\left\|\omega(t)\right\|_{L^p\left(\mathbb{H}\right)} \sim t^{-2+1/p } \I(\omega_0)\quad \mbox{as }t\to \infty,
     \end{equation}  
which is a direct consequence of the asymptotics \eqref{convergencevorticity}.
 For large times, the quadratic terms in the Navier--Stokes equations \eqref{vorticityequation} are thus negligible and  the dynamics of the vorticity are then governed by the linear  equation \eqref{linearizedvorticityequation}.
 
We want to emphasize that this behavior differs significantly from the planar two-dimensional setting, in which the $L^1$ norm is non-increasing, but in general not decaying, for instance, for vorticity distributions of definite sign. Here, the large-time behavior is described by the   Oseen vortices, which are self-similar solutions to the nonlinear planar vorticity equation.

In the general three-dimensional setting, the decay in time of (specified) solutions to the Navier-Stokes equations is known for almost forty years, see for instance \cite{GallagherIftimiePlanchon2003,KajikiyaMiyakawa1986,Kato1984,Schonbeck1985,Wiegner1987}.
In \cite{MiyakawaSchonbeck2001}, Miyakawa and Schonbeck provide optimal decay rates for solutions of Navier--Stokes equations in arbitrary space dimensions. They also contribute a characterization of solutions that satisfy the given optimal decay rate. In \cite{GallayWayne2002}, Gallay and Wayne consider the three-dimensional case and prove the same characterization via invariant manifold techniques
and also detect resonances between the different decay rates.

The purpose of the present paper is to study the large-time asymptotic behavior to \emph{higher order}. More specifically, we
demonstrate how higher-order asymptotics and sharp convergence rates can be obtained, as soon as the convergence to the self-similar profile, or more precisely the decay behavior $\left\|\omega(t)\right\|_{L^p\left(\mathbb{H}\right)} = \mathcal{O}\left(t^{-2+1/p}\right)$, is known. Certainly, our results could also be derived from the deeper and more abstract findings of \cite{GallayWayne2002}. However, what is charming in our new contribution is that
  our proofs are based on fairly simple $L^2$-based entropy methods and do not require any sophisticated or elaborate concepts. 

 Our result applies to situations, in which the initial vorticity is sufficiently concentrated,
 \begin{equation}
 \label{116}
 \int \omega_0(r,z)^2 r e^{\frac{r^2 +z^2}4} d(r,z) <\infty,
 \end{equation}
 which is trivially satisfied for compactly supported configurations, which are the ones that are physically relevant. We will see later on, that this condition is beneficial in the study of the large time behavior.  Notice also that by Jensen's inequality, the bound \eqref{116} entails that $r^2\omega_0\in L^1(\H)$.

In the remaining part of this section, we will give our results on higher-order asymptotic expansions of the vorticity $\omega$. We state these as corollaries because they will be deduced from a rather general result for the relative quantity, Theorem \ref{theoremhoehereOrdnung}. Remember, that we will permanently assume $\omega_0 \in L^1\left(\mathbb{H}\right)$ and $\left\|r^2\omega_0\right\|_{L^1\left(\mathbb{H}\right)}<\infty$.

     \begin{corollary}\label{secondorderasymptoticsforomega}
          Let $\omega$ be a solution of the axisymmetric vorticity equation \eqref{vorticityequation} with initial datum $\omega_0$ such that \eqref{116} holds.  Then there exists a constant $\alpha\in \R$ such that
          \[
          \|\omega(t) - \left(\mathcal{I}(\omega_0) + \frac{\alpha z}{t+1}\right) \zeta_*(t)\|_{L^2(\R^3)} \lesssim \log (t+1)(t+1)^{-9/4}.
          \]
     \end{corollary}
The constant $\alpha$ in the statement 
 can be explicitly computed.      This corollary immediately yields sharp bounds for the first correction of $\omega$, namely
          \begin{align}
               \|\omega(t) - \mathcal{I}\left(\omega_0\right)\zeta_*(t)\|_{L^2(\R^3)}\lesssim (t+1)^{-7/4},
          \end{align}
          for all $t\geq 0$. 
          Now that the first part of our main result is formulated, we like to compare it to already existing works on the asymptotic behavior of solutions of the axisymmetric Navier-Stokes equations without swirl.
          We already mentioned \cite{GallaySverak2015}, where Gallay and \v{S}ver\'ak proved, besides the well-posedness result for initial data in $L^1(\mathbb{H})$, convergence towards the self-similar profile of the linearized equation (for non-negative solutions with finite impulse). They proved the convergence in $L^p\left(\mathbb{H}\right)$ for every $p\in \left[1,\infty\right]$.
Recently, Vila derived a second-order asymptotic expansion for the vorticity in $L^p\left(\mathbb{H}\right)$ for $p \in \left[1,\infty\right]$, see \cite{Vila2022}  and obtained the same convergence rate (in a different norm), including the logarithmic resonance term.
More than twenty years ago, Gallay and Wayne investigated the large-time behavior of solutions to the three-dimensional Navier-Stokes equations without any symmetry assumptions, see \cite{GallayWayne2002}. In this setting they already developed second-order asymptotics for the vorticity and identified the first resonance term. In principle, our results can be obtained from their results by exploiting the symmetry in hindsight, but the invariant manifold techniques they use are much more complicated and elaborate.

     To the best of our knowledge, the higher-order asymptotic expansions in the following corollaries are new for the axisymmetric Navier-Stokes equations without swirl.  We consider first configurations  whose total fluid impulse is trivial. This way, the slowest modes in the frequency spectrum, see Theorem \ref{spectrum} below, will not be excited and we obtain better decay rates.

\begin{namedcorollary}\label{asymptoticsfuerimpulsgleichnull}
    Let $\omega$ be a solution of the axisymmetric vorticity equation \eqref{vorticityequation} with initial datum $\omega_0$ satisfying \eqref{116} and  $\mathcal{I}\left(\omega_0\right)=0$. Then there exist constants $\alpha, \beta$, $\gamma\in \R$ such that
    \begin{align}\mel
         \|\omega(t) 
         -\frac{\alpha z}{t+1}\zeta_*(t)-\frac{1}{t+1}\left(\beta\left(2-\frac{z^2}{t+1}\right)+\gamma\left(8-\frac{r^2}{t+1}\right)\right)\zeta_*\left(t\right)\|_{L^2(\R^3)}\\&
         \lesssim \log(t+1)(t+1)^{-11/4}.
    \end{align}
\end{namedcorollary}
Also the  constants $\beta$ and $\gamma$ can be computed explicitly,
see Theorem \ref{theoremhoehereOrdnung} and the proof of the corollary in the beginning of section \ref{proofs}.

For our last corollary, we additionally demand that the initial datum $\omega_0$ is an even function in the $z$-direction. This geometric assumption annihilates the impact of the next eigenfunction. Thus, no additional correction term arises.
\begin{namedcorollary}\label{asymptoticsfuergeradeangangsdaten}
     Let $\omega$ be given as in Corollary \ref{asymptoticsfuerimpulsgleichnull}. Additionally assume that $\omega_0(r,z)=\omega_0(r,-z)$. Then it holds that
     \begin{align}\mel
          \|\omega(r,z,t) 
     -\frac{1}{t+1}\left(\beta\left(2-\frac{z^2}{t+1}\right)+\gamma\left(8-\frac{r^2}{t+1}\right)\right)\zeta_*\left(t\right)\|_{L^2(\R^3)}\\
     &
     \lesssim \log (t+1)(t+1)^{-13/4}.
     \end{align}
\end{namedcorollary}
In our work \cite{SeisWinkler22}, we imposed similar (and even more extensive) geometric conditions to obtain higher order asymptotic expansions of solutions of the thin-film equation.

In the following subsection, we will transfer equation \eqref{vorticityequation} into a more manageable equation for the relative quantity in self-similar variables.

\subsection{Relative quantity and main result}\label{relativequantity}

The convergence result \eqref{convergencevorticity} motivates the following change of variables into self-similar ones. We define   
     \begin{equation}\label{changeofvariables}
     \begin{aligned}
          &h\left( \hat{t},\hat{r},\hat{z} \right)= (t+1)^2\omega\left(t,r,z\right)  \\ \text{and} \quad
          &m\left(\hat{t},\hat{r},\hat{z}\right) = (t+1)^{3/2} u\left(t,r,z\right)\\
          \text{where} \quad &\hat{t}=\log(t+1), \quad \hat{r}=(t+1)^{-1/2}r,\quad\hat{z}=(t+1)^{-1/2}z.
     \end{aligned}
\end{equation}
In self-similar variables the vorticity equation \eqref{vorticityequation} takes the form
     \begin{align}\label{equationh}
          \partial_{\hat{t}} h =\hat{\Delta} h +\frac{1}{\hat{r}}\partial_{\hat{r}}h+ \frac12 \myvector{\hat{r}\\\hat{z}}\cdot \nabla h +2h-\frac{h}{\hat{r}^2} +e^{-\hat{t}}\left\{ \frac{1}{\hat{r}} m_rh- m\cdot \nabla h\right\}
     \end{align}
and the relation between the vorticity $h$ and the velocity field $m$ remains determined by the linear elliptic system \eqref{relationvoritcityvelocity}. Note that the impulse $\mathcal{I}$ is preserved under the above change of variables, that is 
\begin{equation}\label{115}
\int r^2\omega(t,r,z)  \, drdz  = \int\hat{r}^2 h\left(\hat{r},\hat{z},\hat{t}\right) d\hat{r}d\hat{z},
\end{equation}
 and therefore, the conservation law \eqref{111} is translated into
\begin{equation}
\label{112}
\mathcal{I}(h(\hat t)) = \mathcal{I}(h_0),
\end{equation}
for any $\hat t\ge 0$. 
Hence, 
  the convergence result \eqref{convergencevorticity} can be rephrased in the new variables as 
  \[
  h(\hat t)\to h_*\coloneqq \mathcal{I}\left(h_0\right)\rho_*\quad \mbox{in $L^p\left(\mathbb{H}\right)$ as }\hat t\to \infty.
  \]
In particular, any Lebesgue norm is uniformly controlled, that is
    \begin{align}\label{boundonh}
          \left\|h(\hat t)\right\|_{L^p\left(\mathbb{H}\right)} \leq C   \quad \text{for all }\hat  t >0.
     \end{align}
With these bounds at hand, we will derive sharp rates for the convergence of $h$ towards $h_*$. Our arguments can be iterated to obtain convergence rates for higher order asymptotics as well (under additional assumptions on the initial data).

To do so, it turns out to be useful to work with the relative quantity $f\coloneqq h/\rho_*$. Its evolution is described by
     \begin{align}\label{equationf}
          \partial_{\hat{t}}f+\mathcal{L}f=e^{-\hat{t}}\left\{ \frac12 \myvector{\hat{r}\\\hat{z}}\cdot m f - m\cdot \hat{\nabla} f\right\},
     \end{align}
where 
\[
\mathcal{L}f = -\hat{\Delta} f + \frac12 \myvector{\hat{r}\\\hat{z}} \cdot \nabla f - \frac{3}{\hat{r}} \partial_{\hat{r}}f.
\]
The linear operator $\mathcal{L}$ is symmetric and non-negative in the Hilbert space $L^2\left(\mu\right)=L^2\left(\mathbb{H},\mu \right)$, equipped with the standard scalar product $\ltwoscalarproduct{f}{g} \coloneqq \int fg\, d\mu$, where $d\mu = \frac{1}{16\sqrt{\pi}}\hat{r}^3 e^{-\frac{\hat{r}^2+\hat{z}^2}{4}}d\hat{r}d\hat{z}=\hat{r}^2\rho_*d\hat{r}d\hat{z}$. Indeed, the linear operator can be rewritten as $\mathcal{L}f = -\mu^{-1}\nabla \cdot \left(\mu \nabla f\right)$ and it holds that
     \begin{align}
         \int g\mathcal{L}fd\mu = \int \nabla g \cdot \nabla fd\mu,
     \end{align}
     for sufficiently smooth test functions $f$ and $g$.
Note that $\mu$ is a probability measure.

It follows immediately from the definitions of both the relative variable and the measure, that the mean of $f$ is preserved during the evolution. Indeed, it holds that
\begin{equation}
\label{110}
\la f\ra_{\mu} \coloneqq \int f\, d\mu = \mathcal{I}(h),
\end{equation}
and the fluid impulse is preserved, cf.~\eqref{112}, so that
\begin{equation}
\label{113}
\la f(\hat t)\ra_{\mu} = \la f_0\ra_{\mu},
\end{equation}
for any $\hat t\ge0$.

 The well-posedness result for the vorticity equation \eqref{vorticityequation} ensures the existence of a solution $\omega$ and this solution translates into a solution $f$ of equation \eqref{equationf}. Furthermore, this solution $f$ satisfies $\left\|f(\hat{t})\right\|_{L^1\left(\mu\right)} \le C$ because of the uniform control of $\|r^2 \omega(t)\|_{L^1}$ established in \cite{Vila2022}. 
 Due to the suitably chosen change of variables in \eqref{changeofvariables}, the initial time $t=0$ is mapped to $\hat{t}=0$ and the concentration assumption \eqref{116} on $\omega_0$ is equivalent to the fact that $f_0\in L^2(\mu)$.
An $L^2$ theory for \eqref{equationf} is considered in Lemma~\ref{P1} for completeness.

 A crucial part of our approach consists in explicitly computing the spectrum of $\mathcal{L}$ and of the corresponding eigenfunctions. Even though the equation \eqref{equationf} is nonlinear, the leading order convergence rates of $f$ are given by the eigenvalues of the linear operator. The nonlinear parts can be considered as error terms and can be suitably controlled. In the following theorem, we identify the spectrum of the linear operator and the associated eigenfunctions.

\begin{namedtheorem}[Spectrum of the linear operator]\label{spectrum}
     The spectrum $\sigma(\mathcal{L})$ of the linear operator $\mathcal{L}$ is purely discrete in $L^2\left(\mu\right)$. The eigenvalues are given by
          \begin{align}
               \lambda_{\ell,n} = \ell + \frac{n}{2} \quad \text{ where } \ell,n \in \mathbb{N}_0.
          \end{align}
     The corresponding eigenfunctions are polynomials of the form
          \begin{align}
               \psi_{\ell,n}(r,z)=c_{\ell,n} L_{\ell}^{(1)}\left(\frac{r^2}{4}\right)  H_n\left(\frac{z}{2}\right),
          \end{align}
     where  $H_n$ denotes the $n$-th Hermite polynomial and $L_{\ell}^{(1)}$ the $\ell$-th associated Laguerre polynomial.
     Hence, the eigenvalues $\frac{2m}{2}$ and $\frac{2m+1}{2}$ have multiplicity $m+1$.
     Furthermore, the eigenfunctions $\left\{ \psi_{\ell,n}  : \ell,n \in \N_0\right\}$ form an orthonormal basis of $L^2\left(\mu\right)$, where the constant $c_{\ell,n}$ is merely a normalization factor.
\end{namedtheorem}

We refer to the Chapters 18.1 and 18.3 in \cite{ArfkenHarrisWeber2012} for  definitions and properties of the Hermite and Laguerre polynomials.
 
For the rest of the paper we relabel the eigenvalues in a strictly increasing order, that is $\sigma(\mathcal{L}) = \left\{\lambda_k = \frac{k}{2}:  k\in \N_0\right\}$ with corresponding eigenfunctions $\psi_{k,j}$ for $j \in \left\{1,\dots,N_k\right\}$.

Before stating the main result, we like to illustrate the heuristics of our approach. Equivalently to the convergence of $h$ towards $h_*$, we expect the relative quantity $f$ to converge to its mean $\la f\ra_{\mu}= \mathcal{I}(h_0)$, such that the relative error $f- \la f\ra_{\mu}$ decays to zero. Notice that the limit value $\la f\ra_{\mu}= \mathcal{I}(h_0)$ of $f(t)$  is proportional to the projection $\ltwoscalarproduct{f}{1}$ of $f$ onto the eigenfunction corresponding to the smallest eigenvalue $\lambda_0=0$. As a result, by considering the relative error $f(t)-\la f\ra_{\mu}$ we remove  $\lambda_0$ from the spectrum and we show that the quantity $f(t)-\la f\ra_{\mu}$ decays exponentially fast with rate of the next largest eigenvalue $\lambda_1=1/2$.

In order to improve on the rate of convergence, we eliminate the leading order term by supposing that the fluid impulse $\la f_0\ra_{\mu}$ is zero, cf.~\eqref{110}, \eqref{113}. This way, the relative quantity $f(t)$ will itself vanish asymptotically. The following  abstract result yields the precise leading order asymptotics together with the leading order corrections under the assumption that the first $n$ eigenmodes are not excited.

\begin{theorem}\label{theoremhoehereOrdnung}
Let $n\in\N_0$ be given and      let $f$ be a solution of \eqref{equationf} with initial datum $f_0\in L^2(\mu)$ such that
          \begin{align}\label{assumption}
               \lim \limits_{\hat{t}\to\infty} e^{\lambda_k\hat{t}}\ltwoscalarproduct{f\left(\hat{t}\right) }{\psi_{k,j}} = 0
          \end{align}
     for all $k=0,\dots,n-1$ and $j=1,\dots,N_k$.
     Then there exist  coefficients $a_{n,1},\dots,a_{n,N_n}$ and $a_{n+1,1},\dots,
     a_{n+1,N_{n+1}}$ such that
          \begin{align}
               \ltwomunorm{f\left(\hat{t}\right)- e^{-\lambda_n\hat{t}}\sum_{j=1}^{N_n}a_{n,j}\psi_{n,j}-e^{\lambda_{n+1}\hat{t}}\sum_{j=1}^{N_{n+1}}a_{n+1,j}\psi_{n+1,j}} \lesssim \hat{t}e^{-\lambda_{n+2}\hat{t}} \quad \text{ for all }\hat{t}\geq 0.
          \end{align}
     The coefficients are determined by
     \begin{equation}
     \label{114}
     a_{n,j}\coloneqq \lim \limits_{\hat{t}\to\infty} e^{\lambda_n\hat{t}}\ltwoscalarproduct{f\left(\hat{t}\right)}{\psi_{n,j}}.
     \end{equation}
\end{theorem}

Similar statements were derived earlier, for instance,  in the context of the porous medium equation \cite{Seis15}, the thin film equation \cite{SeisWinkler22}, and the fast diffusion equation \cite{DenzlerKochMcCann15,ChoiMcCannSeis22}.
Notice that the number of correction terms in the theorem is restricted to two. Furthermore, a polynomial factor $\hat{t}$ has to be added on the right-hand side if both correction terms are built-in. It is caused by resonances. These two limitations originate from the distribution of the eigenvalues and the structure of the nonlinear part of \eqref{equationf}. Due to the different scaling, the quadratic part of the vorticity equation \eqref{vorticityequation} is transformed  under the performed change of variables  into the exponentially decaying part of \eqref{equationf}.
More particularly, the nonlinearity does not decay twice as fast as the solution $f$ but, due to the prefactor $e^{-\hat{t}}$, even faster. Indeed, if $f$ decays like $e^{-\lambda \hat{t}}$, we expect the nonlinear part to decay like $e^{-\left(1+\lambda\right)\hat{t}}$. Thus, only correction terms corresponding to eigenvalues $\tilde{\lambda}$ with $\lambda\leq \tilde{\lambda} < 1+\lambda$ can be handled suitably. Because of the structure of the eigenvalues, this limits their number to two. In addition, the resonance term appears since $1+\lambda$ is an eigenvalue whenever $\lambda$ is an eigenvalue.

In the following section, we will turn to the proofs of our results.

\section{Proofs}\label{proofs}
We split  this section into three subsections. In the first one, we  derive the Corollaries \ref{secondorderasymptoticsforomega} to \ref{asymptoticsfuergeradeangangsdaten} from Theorem \ref{spectrum} and Theorem \ref{theoremhoehereOrdnung}. In Subsection \ref{sectionspectrum}, we will compute the spectrum (and the corresponding eigenfunctions) of the linear operator $\mathcal{L}$.
The information about the spectrum of $\mathcal{L}$ finally enables us to prove 
Theorem \ref{theoremhoehereOrdnung} in Subsection \ref{proofsection}.
     
\subsection{Higher order asymptotics --- Proofs of Corollaries \ref{secondorderasymptoticsforomega} to \ref{asymptoticsfuergeradeangangsdaten}}\label{proofsofcorollaries}
     All corollaries of Theorem \ref{theoremhoehereOrdnung} given in the first section will be proved using the same strategy.
We start by considering the leading order asymptotics stated in Corollary \ref{secondorderasymptoticsforomega}   and choose $n=0$ in Theorem \ref{theoremhoehereOrdnung}. In this case, the assumption stated in \eqref{assumption} is actually empty, and because the eigenvalues $\lambda_0=0$ and $\lambda_1=1/2$ have both multiplicity $1$,  it holds that
\[
\|f(\hat t) - a_{0,1}\psi_{0,1} -e^{- \hat t/2} a_{1,1} \psi_{1,1}\|_{L^2(\mu)} \lesssim \hat t e^{-\hat t},
\]
for any $\hat t\ge 1$. Actually, according to Theorem \ref{spectrum}, the eigenfunctions corresponding to the first two eigenvalues take the simple form $\psi_{0,1}=1$ and $\psi_{1,1}=\hat z/\sqrt{2}$, and we infer from \eqref{114}  and the conservation \eqref{113} that
\[
a_{0,1} = \la f_0\ra_{\mu},\quad a_{1,1} = \lim_{\hat t\to \infty} e^{\hat t/2} \la f(\hat{t}),z/\sqrt{2}\ra_{\mu}.
\]
We have noticed in \eqref{110} that the mean of $f$ is nothing but the fluid impulse, which is a conserved (see \eqref{111}) and physically meaningful quantity. All other moments $a_{n,j}$ have no particular relevance, and we will thus not keep track of their origin through \eqref{114}.  
 We may thus rewrite the above leading order asymptotics as
 \[
 \|f(\hat t) - \mathcal{I}(h_0) - e^{-\hat t/2} \alpha\hat z\|_{L^2(\mu)} \lesssim \hat t e^{-\hat t},
 \]
 for some constant $\alpha\in\R$.

In order to rewrite the statement in terms of the original vorticity, we will first go back from the relative variables to the self-similar ones,
\[
\|(h(\hat t)-h_* - e^{-\hat t/2}\alpha\hat z\rho_* )\sqrt{r/\rho_*} \|_{L^2(\R^3)} \lesssim \hat t e^{-\hat t},
\]
recalling that the three-dimensional Lebesgue measure $dx$ reduces to $2\pi r drdz$  when integrated against axisymmetric functions. Observing that  $\sqrt{r/\rho_*} \gtrsim 1$, we may drop this weight. Performing now the change of variables introduced in \eqref{changeofvariables} and using the invariance of the fluid impulse under that transformation \eqref{115}, we arrive at
\[
\|\omega(t) - \mathcal{I}(\omega_0) \zeta_*(t)  - \frac{\alpha z}{t+1} \zeta_*(t)\|_{L^2(\R^3)} \lesssim \log (t+1) (t+1)^{-9/4}.
\]
This proves Corollary \ref{secondorderasymptoticsforomega}.

     Next, we turn to the proof of Corollary \ref{asymptoticsfuerimpulsgleichnull}. Since the impulse $\mathcal{I}\left(\omega(t)\right) = \ltwoscalarproduct{f\left(\hat{t}\right)}{\psi_{0,1}}$ is constant in time, the assumption of Theorem \ref{theoremhoehereOrdnung} in the case $n=1$ is satisfied if $\mathcal{I}\left(\omega_0\right)$ vanishes. 
     We obtain that
     \begin{align}
          \|f(\hat t) -e^{- \hat t/2} a_{1,1} \psi_{1,1} - e^{-\hat{t}}\left(a_{2,1}\psi_{2,1} + a_{2,2}\psi_{2,2}\right)\|_{L^2(\mu)} \lesssim \hat t e^{-3\hat t/2}
     \end{align}
     since the eigenvalue $\lambda_2=1$ has multiplicity two.
     The corresponding eigenfunctions are given by $\psi_{2,1} = (2-\hat{z}^2)/\sqrt{8}$ and $\psi_{2,2}= (8-\hat{r}^2)/\sqrt{32}$, thus identity \eqref{114} yields
     \begin{align}
          a_{2,1} = \lim \limits_{\hat{t}\to \infty}e^{\hat{t}}\ltwoscalarproduct{f(\hat{t})}{(2-\hat{z}^2)/\sqrt{8}}, \quad
          a_{2,2} = \lim \limits_{\hat{t}\to \infty}e^{\hat{t}}\ltwoscalarproduct{f(\hat{t})}{(8-\hat{r}^2)/\sqrt{32}}.
     \end{align}
     From here on, the statement of Corollary \ref{asymptoticsfuerimpulsgleichnull} follows as in the previous proof.

     Lastly, we prove Corollary \ref{asymptoticsfuergeradeangangsdaten}. By assumption, the initial data $\omega_0$ is even in the $z$-variable. Since the axisymmetric vorticity equation \eqref{vorticityequation} is invariant under reflections in the $z$-variable, uniqueness of the solution $\omega(t)$ to \eqref{vorticityequation} guarantees that it inherits the same symmetry property from its initial datum. We immediately deduce that
     \begin{align}
          a_{1,1} =\lim \limits_{\hat{t}\to\infty} e^{\hat{t}/2}\ltwoscalarproduct{f\left(\hat{t}\right)}{\hat{z}/\sqrt{2}} =\lim \limits_{t\to\infty} \frac{1}{\sqrt{2}}\int zr^2\omega(t)drdz = 0,
     \end{align}
     since the term $\int zr^2\omega(t) drdz$ vanishes for all $t\geq0$ due to the symmetry of $\omega$. This allows us  to apply Theorem \ref{theoremhoehereOrdnung} in the case $n=2$ and the proof proceeds in the established way. Note that the same reasoning as for $a_{1,1}$ shows that the constants
     \begin{align}
          a_{3,1}\coloneqq \lim \limits_{\hat{t}\to\infty} e^{3\hat{t}/2}\ltwoscalarproduct{f\left(\hat{t}\right)}{\psi_{3,1}} \quad \text{and} \quad a_{3,2}\coloneqq \lim \limits_{\hat{t}\to\infty} e^{3\hat{t}/2}\ltwoscalarproduct{f\left(\hat{t}\right)}{\psi_{3,2}}
     \end{align}
     both vanish as well, since $\psi_{3,1}\sim 6\hat{z}-\hat{z}^3$ and $\psi_{3,2}\sim\hat{z}(8-\hat{r}^2)$ are odd in the $z$-direction. Therefore, no second correction term appears in the statement of Corollary \ref{asymptoticsfuergeradeangangsdaten}. We see  that the symmetry assumption on $\omega_0$ does not only rule out the impact of the eigenfunction $\psi_{1,1}$ but also of any other odd eigenfunction.
     
     \medskip

     After deriving the main results for the vorticity $\omega$, we now investigate the spectrum of the linear operator $\mathcal{L}$.

\subsection{The spectrum of the linear operator \texorpdfstring{$\mathcal{L}$}{} --- Theorem \ref{spectrum}}\label{sectionspectrum}

In this subsection, our aim is to diagonalize the differential operator $\L$ in the Hilbert space $L^2(\mu)$.
Determining the eigenvalues of a linear differential operator is a classical task in mathematical physics and, more generally, the theory of (non-)linear partial differential equations. It is often convenient, to exploit the symmetries of the operator to reduce the dimensionality of the problem. This can be done by making a separation-of-variables ansatz. In our setting, we decompose the 
linear operator $\mathcal{L}$ into its  radial and axial parts, that is,
\[
\L = \L_r + \L_z  \quad \mbox{where }\mathcal{L}_r = -\partial_r^2+\frac{1}{2}r\partial_r - \frac{3}{r}\partial_r,\quad \mathcal{L}_z=-\partial_z^2 +\frac{1}{2}z\partial_z,
\]
and we will look for eigenfunctions of the form $f(r,z) = \phi(r)\psi(z)$.

It is convenient to rescale variables in order to transform the differential operators into a kind-of ``textbook'' form. More precisely, we introduce
\[
s = \frac{r^2}4 \quad \mbox{and}\quad y = \frac{z}2,
\]
so that
\[
\L_s = -s\partial_s^2 +(s-2)\partial_s,\quad  \L_y = -\frac14\partial_y^2 +\frac12 y\partial_y.
\]
The corresponding eigenvalue equations are  Laguerre and Hermite differential equations, respectively. For any $\ell,n\in\N_0$, it thus holds
\begin{equation}\label{3}
\L_s L_{\ell}^{(1)} = \ell L^{(1)}_{\ell}, \quad \L_y H_n = \frac{n}2 H_n,
\end{equation}
where $L_{\ell}^{(1)}$ are Laguerre polynomials and $H_n$ are Hermite polynomials. These form an orthogonal basis of $L^2(\R_+,se^{-s})$ and $L^2(\R, e^{-y^2})$, respectively. We refer to Chapters 18.1 and 18.3 in the monograph \cite{ArfkenHarrisWeber2012} for references.

This observation motivates to consider and   orthonormal basis of  $L^2(\mu)$ given by
\[
\phi_\ell(r) \otimes\psi_{n}(z) = c_{n,\ell} L_{\ell}^{(1)} (r^2/4)  H_n(z/2),
\]
for some constants $c_{n,\ell}$ suitably chosen, so that
\[
f(r,z) = \sum_{\ell,r}\langle f,\phi_{\ell}\otimes \psi_n\rangle \phi_{\ell}(r)\psi_n(z).
\]
The eigenvalues of $\L$ are then precisely the eigenvalues $\phi_{\ell}(r)\otimes\psi_n(z)$, which can be computed by using \eqref{3}, namely,
\[
\L(\phi_{\ell}\otimes \psi_n)= \psi_n \L_r\phi_{\ell} + \phi_{\ell}\L_z \psi_n = \left(\ell+\frac{n}2\right)\phi_{\ell}\otimes \psi_n.
\]

With an orthonormal basis consisting of eigenfunctions $\psi_{\ell,n}\coloneqq \phi_\ell \otimes \psi_n$ at hand, we can rewrite $\mathcal{L}$ as follows
\begin{align}
     \mathcal{L}f = \sum\limits_{n,\ell=0}^\infty \lambda_{\ell,n} \ltwoscalarproduct{f}{\psi_{\ell,n}}\psi_{\ell,n},
\end{align}
where $\lambda_{\ell,n}\coloneqq \ell + \frac{n}{2}$.
By standard results from functional analysis, see, e.g., Chapter X.1 in \cite{Conway1990}, an operator of this form is self-adjoint and its spectrum is given by 
\[
\sigma\left(\mathcal{
L}\right) = \overline{\left\{\lambda_{\ell,n}\right\}} = \left\{\ell + \frac{n}{2} : \ell,n \in \N_0\right\}.
\]
The results in Theorem \ref{spectrum} are thus derived.

\medskip

An important tool for $L^2$-based proofs is the Poincar\'e inequality which enables us to estimate the Dirichlet form of the operator $\mathcal{L}$, that is the $L^2(\mu)$ norm of the gradient. Its standard version in the Gauss space is the Brascamp--Lieb inequality
                    \begin{align}\label{2}
                     \frac{1}{2}\ltwomunorm{f -\la f \ra_{\mu}}^2\le     \ltwomunorm{\nabla f(t)}^2 ,                 
                     \end{align} 
  where $\la f\ra_{\mu}$ denotes the integral of $f$ with respect to the probability measure $\mu$,                 see, for instance, \cite{BrascampLieb1976,BobkovLedoux2000}. Having the precise information on the  spectrum at hand, we are able to compute the sharp constant easily. Indeed, setting $g=f-\la f\ra_\mu$ for convenience, we use the eigenvalue expansion of $g$ to write
  \begin{align*}
  \|\grad g\|_{L^2(\mu)}^2 = \la g, \mathcal{L} g\ra_{\mu}  = \sum_{n=0}^{\infty} \sum_{j=1}^{N_n} \la g,\mathcal{L}\psi_{n,j}\ra_{\mu}\la g,\psi_{n,j}\ra_{\mu} = \sum_{n=0}^{\infty} \lambda_n \sum_{j=1}^{N_n} \la g,\psi_{n,j}\ra_{\mu}^2.
  \end{align*}
 Because $\psi_{0,1}\equiv 1$ and since $g$ has zero mean, the zeroth mode can be neglected. For all other modes, we bound $\lambda_n\ge \lambda_1 = 1/2$, so that                  
   \begin{align*}
  \|\grad g\|_{L^2(\mu)}^2 \ge \frac12 \sum_{n=0}^{\infty}\sum_{j=1}^{N_n} \la g,\psi_{n,j}\ra_{\mu}^2 = \frac{1}{2}\|g\|_{L^2(\mu)}^2,
  \end{align*}
which is just \eqref{2}.

\subsection{Main result for relative quantity --- Proof of Theorem \ref{theoremhoehereOrdnung}}\label{proofsection}
     In the remaining part, we only work with the equations for $f$ or $h$. For notational convenience we thus drop the hats of the spatial and time variables.

     As usual in $L^2$-theory based arguments, we will test solutions $f$ of \eqref{equationf} against other functions like eigenfunctions or the solution itself. While the linear part of the equation behaves quite well in this situation, the nonlinear part has to be treated with more care. For convenience, we state an integration by parts identity for \emph{arbitrary} functions $f$ and $g$, that will be used several times in the sequel: It holds that
               \begin{equation}\label{integrationbypartsidentity}
            \frac{1}{2}   \int g \myvector{r\\z}\cdot m f d\mu - \int g m\cdot \nabla f d\mu 
               = \int \nabla g\cdot m f d\mu +2 \int  \frac{m_r}{r}gf d\mu,
          \end{equation}
where      we used the explicit form of the measure $\mu$ and the fact that the divergence-free condition of the velocity field $m$ in cylindrical coordinates is equivalent to $\nabla\cdot \left(rm\right)=0$. 

We need an estimate on the (rescaled) velocity vector field.

     \begin{lemma}\label{estimatem}
          Let $\omega$ be a solution  with initial data $\omega_0$ as in Theorem \ref{TA} and let $m$   denote the corresponding rescaled  velocity field  \eqref{changeofvariables}.  Then it holds that
               \begin{align}
                    \left\|m(t)\right\|_{L^\infty\left(\mathbb{H}\right)}\lesssim 1 \quad \text{ for all } t\geq 0.
               \end{align}
     \end{lemma}
     
     \begin{proof}The estimate relies on an interpolation estimate derived by Gallay  and \v{S}ver\'ak in \cite{GallaySverak2015}  and a priori estimates on the rescaled vorticity field.
     Indeed, 
          the  estimate is a direct consequence of an interpolation in Proposition 2.3  from \cite{GallaySverak2015},
                         \begin{align}
                    \left\| m \right\|_{L^\infty\left(\mathbb{H}\right)} \lesssim \|h\|_{L^1\left(\mathbb{H}\right)}^{1/2} \|h\|_{L^{\infty}\left(\mathbb{H}\right)}^{1/2},
               \end{align}
where $h$ is the vorticity corresponding to $m$, and from the uniform bounds in \eqref{boundonh}.
     \end{proof}

We use the velocity bound to ensure that $L^2(\mu)$ is a reasonable solution space.

\begin{lemma}
\label{P1}
Suppose that $f$ is a solution to \eqref{equationf} with initial data $f_0\in L^2(\mu)$. Then $f(t)\in L^2(\mu)$ for any $t$ and it holds that
\[
\|f(t) - \la f\ra_{\mu}\|_{L^2(\mu)} \lesssim  e^{-t/2}.
\]
In particular, $f(t)$ is uniformly bounded in $L^2(\mu)$.
\end{lemma}

\begin{proof}
The proof of this estimate is fairly standard. We recall that the fluid impulse $\la f\ra_{\mu}$ is preserved \eqref{113}, and we compute, integrating by parts twice, using \eqref{integrationbypartsidentity},
\begin{align*}
\frac12\frac{d}{dt} \|f-\la f\ra_{\mu}\|_{L^2(\mu)}^2 + \|\grad f\|_{L^2(\mu)}^2= e^{-t} \left(\la f-\la f\ra_{\mu} , m\cdot \grad f\ra_{\mu} +2 \la f-\la f\ra_{\mu} ,r^{-1} m_r f\ra_{\mu}\right).
\end{align*}
We estimate the velocity vector uniformly with the help of Lemma \ref{estimatem}, and invoke the Cauchy--Schwarz inequality to obtain
\begin{align*}
\mel \frac12\frac{d}{dt} \|f-\la f\ra_{\mu}\|_{L^2(\mu)}^2 + \|\grad f\|_{L^2(\mu)}^2\\
& \lesssim e^{-t} \|f-\la f\ra_{\mu}\|_{L^2(\mu)} \left( \|\grad f\|_{L^2(\mu)}   + |\la f\ra_{\mu}| + \|r^{-1}(f-\la f\ra_{\mu})\|_{L^2(\mu)}\right).
\end{align*}
Via the Hardy inequality from Lemma \ref{L3} in the appendix and the Poincar\'e inequality in \eqref{2}, we notice that  the last error term above can be dropped. Using Young's inequality, \eqref{110} and the conservation of the fluid impulse, we thus arrive at
\[
 \frac{1}{2}\frac{d}{dt} \|f-\la f\ra_{\mu}\|_{L^2(\mu)}^2 + (1-C e^{-t}) \|\grad f\|_{L^2(\mu)}^2 \le C e^{-t} \left(\|f-\la f\ra_{\mu}\|_{L^2(\mu)}+ \|f-\la f\ra_{\mu}\|_{L^2(\mu)}^2\right).
 \]
Applying the Poincar\'e inequality \eqref{2} once more, we obtain an estimate that is of the form considered in part \ref{gronwallB} of the Gronwall lemma \ref{gronwallversion} with $\lambda=1/2$ and $\mu=1$. Integration yields thus the first statement of the proposition. 
With help of the triangle inequality we immediately obtain
\begin{align}
     \ltwomunorm{f(t)} \lesssim e^{-t/2} + |\la f\ra_{\mu} |\lesssim 1
\end{align}
since $\langle f \rangle_\mu $ is constant in time, cf.\ \eqref{113}.
\end{proof}

The above lemma does not only legitimize $L^2(\mu)$ as a solution space, but also serves as the starting point for higher order asymptotics. In Proposition \ref{decayoff}, we will derive faster decay rates of $f$ with help of a dynamically improved Poincar\'e inequality via induction over the eigenvalues. Lemma \ref{P1} will constitute the base case. 

 In our derivation of Theorem \ref{theoremhoehereOrdnung}, it will be important to ensure that low frequency eigenmodes are asymptotically suppressed. This will be guaranteed if the solution is decaying fast enough. Indeed, the following result shows that, up to higher order error terms, the spectral gap estimate can be improved along the nonlinear evolution as in the linear case, cf.~\eqref{2}.

\begin{namedproposition}[Dynamically improved Poincar\'e inequality]\label{dynamicpoincare}
     Let $f$ be a solution of \eqref{equationf} with initial datum $f_0 \in L^2(\mu)$ satisfying
          \begin{align}\label{105}
               \ltwomunorm{f(t)}\lesssim e^{-\lambda_nt} \quad \text{ for all } t\geq 0
          \end{align}
     and a given $n\in\N_0$.
     Then the following two inequalities hold.
     \begin{enumerate}[label=(\alph{enumi}),ref=(\alph{enumi})]
          \item \label{poincareA}There exists a constant $C\geq 0$ such that
               \begin{align}
            \mel        \lambda_{n+1}\| f(t) - e^{-\lambda_nt}\sum_{j=1}^{N_n}a_{n,j}\psi_{n,j}\|_{L^2(\mu)}^2\\
            & \le   \|\nabla (f(t)-e^{-\lambda_nt}\sum_{j=1}^{N_n}a_{n,j}\psi_{n,k})\|_{L^2(\mu)}^2 + C e^{-2\lambda_{n+2}t}
               \end{align}
               holds for all $t\geq0$. Here $a_{n,j}\coloneqq \lim \limits_{t\to\infty} e^{\lambda_nt}\ltwoscalarproduct{f(t)}{\psi_{n,j}}$.
          \item \label{poincareB}There exists a constant $C\geq 0$ such that
               \begin{align}\mel
       \lambda_{n+2}\|f(t) - e^{-\lambda_nt}\sum_{j=1}^{N_n}a_{n,j}\psi_{n,j}- e^{-\lambda_{n+1}t}\sum_{j=1}^{N_{n+1}}a_{n+1,j}\psi_{n+1,j}\|_{L^2(\mu)}^2    \\
       &\le           \|\nabla (f(t)-e^{-\lambda_nt}\sum_{j=1}^{N_n}a_{n,j}\psi_{n,k}-e^{-\lambda_{n+1}t}\sum_{j=1}^{N_{n+1}}a_{n+1,j}\psi_{n+1,k})\|_{L^2(\mu)}^2\\
       &\quad                     + Ce^{-2\lambda_{n+2}t}
               \end{align}
               holds for all $t\geq0$.  The constants $a_{n,j}$ are given as in part \ref{poincareA} and accordingly $a_{n+1,j}\coloneqq \lim\limits_{t\to\infty} e^{\lambda_{n+1}t}\ltwoscalarproduct{f(t)}{\psi_{n+1,j}}$.
     \end{enumerate}
     
\end{namedproposition}

Note  that \emph{a priori} it is not clear that the constants $a_{n,j}$ and $a_{n+1,j}$ are finite. 

\begin{proof}
We commence with the proof of part \ref{poincareA}. As a start, we expand $f$ in eigenfunctions, i.e., $f = \sum_{k=0}^{\infty}\sum_{j=1}^{N_k} \ltwoscalarproduct{f}{\psi_{k,j}}\psi_{k,j}$ and compute, using the orthonormality of the eigenfunctions that
\begin{equation}\label{100}
 \|f-e^{-\lambda_n t} \sum_j a_{n,j} \psi_{n,j}\|_{L^2(\mu)}^2= \sum_{m,j}  \la f,\psi_{m,j}\ra_{\mu}^2 -2  e^{-\lambda_n t}\sum_j a_{n,j} \la f, \psi_{n,j}\ra_{\mu} +e^{-2\lambda_nt } \sum_j a_{n,j}^2.
\end{equation}
Similarly, using in addition the symmetry of the linear operator $\mathcal{L}$, we notice that
\begin{align*}
\mel \|\grad(f-e^{-\lambda_n t} \sum_j a_{n,j} \psi_{n,j})\|_{L^2(\mu)}^2\\
& = \sum_{m,j} \lambda_m \la f,\psi_{m,j}\ra_{\mu}^2 -2 \lambda_n e^{-\lambda_n t}\sum_j a_{n,j} \la f, \psi_{n,j}\ra_{\mu} +e^{-2\lambda_nt }\lambda_n \sum_j a_{n,j}^2.
\end{align*}
For the first term on the right-hand side, we have that
\begin{align*}
\sum_{m,j} \lambda_m \la f,\psi_{m,j}\ra_{\mu}^2 & \ge \lambda_{n+1} \sum_{m,j} \la f,\psi_{m,j}\ra_{\mu}^2  - \sum_{m=0}^n \sum_j (\lambda_{n+1} - \lambda_{m})\la f,\psi_{m,j}\ra_{\mu}^2,
\end{align*}
and using the identity \eqref{100}, we then find
\begin{equation}\label{103}
\begin{aligned}
\mel \|\grad(f-e^{-\lambda_n t} \sum_j a_{n,j} \psi_{n,j})\|_{L^2(\mu)}^2\\
&\ge \lambda_{n+1}  \|f-e^{-\lambda_n t} \sum_j a_{n,j} \psi_{n,j}\|_{L^2(\mu)}^2 - \sum_{m=0}^n \sum_j (\lambda_{n+1} - \lambda_{m})\la f,\psi_{m,j}\ra_{\mu}^2\\
&\quad +2(\lambda_{n+1}- \lambda_n )e^{-\lambda_n t}\sum_j a_{n,j} \la f, \psi_{n,j}\ra_{\mu} - e^{-2\lambda_nt }(\lambda_{n+1}-\lambda_n )\sum_j a_{n,j}^2\\
& = \lambda_{n+1}  \|f-e^{-\lambda_n t} \sum_j a_{n,j} \psi_{n,j}\|_{L^2(\mu)}^2 - \sum_{m=0}^{n-1} \sum_j (\lambda_{n+1} - \lambda_{m})\la f,\psi_{m,j}\ra_{\mu}^2\\
&\quad - (\lambda_{n+1}-\lambda_n) \sum_j \left(\la f,\psi_{n,j}\ra_{\mu} -e^{-\lambda_n t} a_{n,j}\right)^2.
\end{aligned}
\end{equation}
We have to ensure that the second and the third term on the right-hand side are of the order $O(e^{-2\lambda_{n+2}t})$.

Let us first consider the evolution of the projections $\ltwoscalarproduct{f(t)}{\psi_{m,j}}$ for $m < n$. In view of the imposed decay on $f$, it is clear that they are decaying faster than $e^{-\lambda_m t}$. Indeed, it holds that
\begin{equation}\label{101}
e^{\lambda_m t} |\la f(t),\psi_{m,j}\ra_{\mu} | \le e^{\lambda_m t} \|f(t)\|_{L^2(\mu)} \lesssim e^{-(\lambda_n-\lambda_m)t},
\end{equation}
by the assumption in the lemma, and the right-hand side is decaying to zero for $m<n$. In order to improve on the rate, we use the equation. Applying the projection to \eqref{equationf} yields
\begin{align}
\frac{d}{dt} \la f,\psi_{m,j}\ra_{\mu} + \lambda_m \la f,\psi_{m,j}\ra_{\mu} = \frac{1}{2} e^{-t} \la \myvector{r\\z} \cdot m f,\psi_{m,j}\ra_{\mu}  -e^{-t} \la m\cdot\grad f,\psi_{m,j}\ra_{\mu}.
\end{align}
Using the integration by parts formula \eqref{integrationbypartsidentity}, the derivatives can be put onto the eigenfunction,
\[
\frac{d}{dt} \la f,\psi_{m,j}\ra_{\mu} + \lambda_m \la f,\psi_{m,j}\ra_{\mu} = e^{-t} \la f,\grad\psi_{m,j}\cdot m\ra_{\mu} +2 e^{-t} \la f, r^{-1} m_r\psi_{m,j}\ra_{\mu}.
\]
Using the decay assumption on the solution, we thus obtain the estimate
\[
\left| \frac{d}{dt} \la f,\psi_{m,j}\ra_{\mu} + \lambda_m \la f,\psi_{m,j}\ra_{\mu}\right|\lesssim  e^{-(\lambda_n+1)t} \|m\|_{L^{\infty}}\left(\|\grad\psi_{m,j}\|_{L^2(\mu)} + \|r^{-1} \psi_{m,j}\|_{L^2(\mu)}\right).
\]
Thanks to the Hardy-type inequality from Lemma \ref{L3}, the terms involving the eigenfunctions are controlled by $\|\psi_{m,j}\|_{L^2(\mu)} + \|\grad\psi_{m,j}\|_{L^2(\mu)}=1+\sqrt{\lambda_m}$. Moreover, the velocity is bounded thanks to Lemma \ref{estimatem}. We thus find that
\begin{equation}\label{104}
\left| \frac{d}{dt} \left( e^{\lambda_m t} \la f,\psi_{m,j}\ra_{\mu}\right)  \right| \lesssim 
e^{-(\lambda_n-\lambda_m+1)t} .
\end{equation}

If $m<n$, we use the observation \eqref{101} to   deduce that 
\begin{equation}\label{102}
 | \la f(t),\psi_{m,j}\ra_{\mu}|\lesssim e^{-(\lambda_n+1)t} =e^{-\lambda_{n+2}t}.
 \end{equation}
 
If $m=n$, we observe that 
  $t\mapsto \frac{d}{dt}\absolutevalue{e^{\lambda_nt}\ltwoscalarproduct{f(t)}{\psi_{n,j}}}$ is a Cauchy sequence and therefore there exists a constant $a_{n,j}$, indeed given by $a_{n,j} = \lim\limits_{t\to\infty}e^{\lambda_nt}\ltwoscalarproduct{f(t)}{\psi_{n,j}}$, such that \[
     \absolutevalue{e^{\lambda_nt}\ltwoscalarproduct{f(t)}{\psi_{n,j}} - a_{n,j}}\lesssim e^{-t}.
     \]
Multiplying this estimate by $e^{-\lambda_n t}$, we obtain the same decay behavior as in \eqref{102}, and thus, plugging both bounds into \eqref{103} finishes the proof of   part \ref{poincareA}.

     The proof of part \ref{poincareB} proceeds essentially in the  same way. It holds that
     \begin{align*}
\mel     \lambda_{n+2} \|f - e^{-\lambda_n t}\sum_j a_{n,j} \psi_{n,j} - e^{-\lambda_{n+1}t }\sum_j a_{n+1,j}\psi_{n+1,j}\|_{L^2(\mu)}^2\\
&\le \|\grad(f - e^{-\lambda_n t}\sum_j a_{n,j} \psi_{n,j} - e^{-\lambda_{n+1}t }\sum_j a_{n+1,j}\psi_{n+1,j})\|_{L^2(\mu)}^2\\
&\quad + \sum_{m=0}^{n-1} \sum_j (\lambda_{n+2}-\lambda_m) \la f,\psi_{m,j}\ra_{\mu}^2 \\
&\quad + (\lambda_{n+2}-\lambda_n) \sum_j \left(\la f,\psi_{n,j}\ra_{\mu} - e^{-\lambda_{n}t}a_{n,j}\right)^2\\
&\quad + (\lambda_{n+2}-\lambda_n)\sum_j \left(\la f,\psi_{n+1,j}\ra_{\mu} - e^{-\lambda_{n+1} t} a_{n+1,j}\right)^2.
     \end{align*}
     The only term that is new compared to the case considered in \ref{poincareA} is the one displayed in the last line above. For this, we note that  the estimate in \eqref{104} (where now $m=n+1$) becomes
     \[
     \left|\frac{d}{dt} \left(e^{\lambda_{n+1}t}\la f,\psi_{n+1,j}\ra_{\mu}\right)\right| \lesssim e^{-(\lambda_n-\lambda_{n+1}+1)t} = e^{-t/2},
     \]
and from this point on, we may conclude as before.
\end{proof}

We like to remark  that  the maximal number of correction terms for the dynamically improved Poincar\'e inequality is two --- at least when relying upon our  elementary proof.
This observation provides another explanation for the constraint on the number of correction terms in the main Theorem \ref{theoremhoehereOrdnung}.

The following proposition illustrates why the decaying hypothesis in the previous proposition can actually be assumed to prove our main result, although the assumptions in Theorem \ref{theoremhoehereOrdnung}   seem to be weaker.

\begin{proposition}\label{decayoff}
     Let $n\in\N_0$ and $f$ a solution of \eqref{equationf} with initial datum $f_0 \in L^2(\mu)$ such that
          \begin{align}\label{assumption2}
               \lim \limits_{t\to\infty} e^{\lambda_kt}\ltwoscalarproduct{f(t) }{\psi_{k,j}} = 0
          \end{align}
     for all $k=0,\dots,n-1$ and $j=1,\dots,N_k$. Then it holds that
          \begin{align}
               \ltwomunorm{f(t)}\lesssim e^{-\lambda_nt} \quad \text{ for all }t\geq 0.
          \end{align}
\end{proposition}

\begin{proof}
     We prove the statement via induction over $n$. In the base case $n=0$, assumption \eqref{assumption2} is empty and $\lambda_0=0$. The statement $\ltwomunorm{f(t)}\lesssim 1$ is covered by Lemma \ref{P1} since $f_0\in L^2(\mu)$.
 
   Let us now assume that \eqref{assumption2} holds true for all $k\leq n$ and in addition that
          \begin{align}\label{9}
               \ltwomunorm{f(t)}\lesssim e^{-\lambda_nt}
          \end{align}
by induction hypothesis. Our goal is to improve the exponential decay rate from $\lambda_n$ to $\lambda_{n+1}$.
The argument is very similar to that of Lemma \ref{P1}, which we have to improve.  
With estimate \eqref{9} and the assumption \eqref{assumption2} at hand, the Poincar\'e inequality Lemma \ref{dynamicpoincare}, part \ref{poincareA}, provides
          \begin{align}\label{10}
        \lambda_{n+1}\ltwomunorm{f(t)}^2 \le       \ltwomunorm{\nabla f(t)}^2   + Ce^{-2\lambda_{n+2}t}.
          \end{align}
          
     We now compute the $L^2$ decay of $f$ by testing the evolution and integrating by parts, 
          \begin{align}
               \frac{1}{2}\frac{d}{dt} \ltwomunorm{f}^2  + \|\grad f\|_{L^2(\mu)}^2 & = e^{-t}\left(\int \frac{1}{2}\myvector{r\\z} \cdot m f^2 d\mu - \int f m\cdot \nabla f d\mu\right).          \end{align}
          Using the  integration-by-parts identity \eqref{integrationbypartsidentity}, the latter turns into
          \[
     \frac{d}{dt} \ltwomunorm{f}^2   +2\| \grad f\|_{L^2(\mu)}^2\le 2e^{-t}\left(\int  f m\cdot \nabla f  d\mu + 2 \int  \frac{m_r}{r}f^2d\mu\right).
     \]
We now estimate the right-hand side with the help of the velocity bound in  Lemma \ref{estimatem} and the Hardy-type inequality in Lemma \ref{L3},
\begin{align*}
\left|\int  f m\cdot \nabla f  d\mu + 2 \int  \frac{m_r}{r}f^2d\mu \right| &\lesssim \|m\|_{L^{\infty}} \left( \|f\|_{L^2(\mu)} \|\grad f\|_{L^2(\mu)} + \|r^{-1} f\|_{L^2(\mu)}\|f\|_{L^2(\mu)}\right)\\
&\lesssim \|f\|_{L^2(\mu)}^2 + \|f\|_{L^2(\mu)} \|\grad f\|_{L^2(\mu)},
\end{align*}
so that via Young's and  Poincar\'e's estimate \eqref{10}
\[
\frac{d}{dt} \ltwomunorm{f}^2   +(2\lambda_{n+1}-C e^{-t}) \| f\|_{L^2(\mu)}^2 \lesssim e^{-2\lambda_{n+2}t}\]
for some $C>0$.
We invoke the Gronwall inequality from part \ref{gronwallB} of Lemma \ref{gronwallversion} to deduce 
\[
\|f(t)\|_{L^2(\mu)}\lesssim e^{-\lambda_{n+1}t}
\]
as desired.
\end{proof}
     
Finally, we are in the position to prove our main result.

\begin{proof}[ of Theorem \ref{theoremhoehereOrdnung}]We start by noticing that, under the hypothesis of the theorem, Propositions \ref{decayoff} and \ref{dynamicpoincare} imply that
\begin{equation}
\label{106}
\lambda_{n+2} \|f (t) - \psi_n(t)-\psi_{n+1}(t)\|_{L^2(\mu)}^2 \le \|\grad(f(t)-\psi_n(t)-\psi_{n+1}(t))\|_{L^2(\mu)}^2 + C e^{-2\lambda_{n+2}t},
\end{equation}
where we have set
\[
\psi_m(t) = e^{-\lambda_m t} \sum_{j=1}^{N_m} a_{m,j}\psi_{m,j},
\]
for $m=n,n+1$. Noticing that the  $\psi_m$'s solve the linear equation $\partial_t \psi_m + \mathcal{L}\psi_m=0$, we  have that 
          \begin{align}\mel
               \partial_t \left(f-\psi_n - \psi_{n+1}\right) + \mathcal{L}\left(f-\psi_n - \psi_{n+1}\right) 
               \\
              & = e^{-t}\left\{\frac{1}{2}\myvector{r\\z}\cdot m \left(f-\psi_n - \psi_{n+1}\right)-m\cdot \nabla \left(f-\psi_n - \psi_{n+1}\right)   \right.
               \\&
                \quad  + \left.\frac{1}{2}\myvector{r\\z}\cdot m \left(\psi_n + \psi_{n+1}\right)  - m\cdot \nabla \left(\psi_n + \psi_{n+1}\right)\right\}.
          \end{align}
          Using the symmetry of $\mathcal{L}$ and the integration-by-parts identity \eqref{integrationbypartsidentity}, we consequentially compute
          \begin{align}\mel
               \frac{1}{2}\frac{d}{dt} \ltwomunorm{f-\psi_n-\psi_{n+1}}^2  + \ltwomunorm{\grad(f-\psi_n-\psi_{n+1})}^2 
            \\
            &   =                
                e^{-t}\left\{\frac{1}{2}\int  \left(f-\psi_n-\psi_{n+1}\right)^2\myvector{r\\z}\cdot m d\mu \right.
                       \\&
               \qquad -\int  \left(f-\psi_n-\psi_{n+1}\right) m\cdot \nabla  \left(f-\psi_n-\psi_{n+1}\right)d\mu 
               \\&
               \qquad +\frac{1}{2}\int  \left(f-\psi_n-\psi_{n+1}\right) \myvector{r\\z} \cdot m  \left(\psi_n+\psi_{n+1}\right)d\mu
               \\&
               \qquad\left.- \int  \left(f-\psi_n-\psi_{n+1}\right) m\cdot \nabla  \left(\psi_n+\psi_{n+1}\right)d\mu\right\}
               \\& =e^{-t}\left\{\int (f-\psi_n-\psi_{n+1}) m\cdot \grad(f-\psi_n-\psi_{n+1})\, d\mu\right.\\
               &\qquad + 2\int \frac{m_r}r (f-\psi_n-\psi_{n+1})^2\, d\mu \\
              &\qquad + \int (\psi_n+\psi_{n+1}) m\cdot \grad(f-\psi_n-\psi_{n+1})\, d\mu\\
              &\qquad +\left. \int \frac{m_r}r (\psi_n+\psi_{n+1}) (f-\psi_n-\psi_{n+1})\, d\mu\right\}.
          \end{align}
          We may uniformly estimate the velocity term with the help of Lemma \ref{estimatem}. Making use of the Hardy-type inequality from Lemma \ref{L3} and using the trivial bounds
          \[
          \|\psi_n+\psi_{n+1}\|_{L^2(\mu)} + \|\grad(\psi_n+\psi_{n+1})\|_{L^2(\mu)} \lesssim e^{-\lambda_n t},
          \]
          we find that
       \begin{align}\mel
                \frac{d}{dt} \ltwomunorm{f-\psi_n-\psi_{n+1}}^2  + 2\ltwomunorm{\grad(f-\psi_n-\psi_{n+1})}^2 
            \\      
   &\lesssim e^{-t}\left\{ \|f-\psi_n-\psi_{n+1}\|_{L^2(\mu)} \|\grad(f-\psi_n-\psi_{n+1}) \|_{L^2(\mu)}      +   \|f-\psi_n-\psi_{n+1}\|_{L^2(\mu)}^2\right\}\\
   &\quad  + e^{-\lambda_{n+2} t} \left\{ \|f-\psi_n-\psi_{n+1}\|_{L^2(\mu)} +   \|\grad(f-\psi_n-\psi_{n+1})\|_{L^2(\mu)} \right\}.
   \end{align}
 An application of Young's inequality followed by the Poincar\'e estimate \eqref{106} then gives
 \[
 \frac{d}{dt}    \ltwomunorm{f-\psi_n-\psi_{n+1}}^2     + (2\lambda_{n+2} - C e^{-t}) \ltwomunorm{f-\psi_n-\psi_{n+1}}^2  \lesssim e^{-2\lambda_{n+2}t}.
 \]  
  The statement of the theorem now directly follows with from the Gronwall inequality in Proposition \ref{gronwallversion} part \ref{gronwallC}.
\end{proof}

\newpage
\appendix
\renewcommand{\thesection}{\Alph{section}}
\section*{Appendix --- Some elementary estimates}\label{appendix}
\setcounter{section}{1}
\setcounter{theorem}{0}
In this appendix we collect some rather standard result, whose proofs we provide for the convenience of the reader.

\begin{namedlemma}[Hardy-type inequality]\label{L3}
It holds that
\[
\|\psi\|_{L^2(\rho_*)} \lesssim \|\psi\|_{L^2(\mu)} + \|\partial_r \psi\|_{L^2(\mu)}.
\]
\end{namedlemma}

\begin{proof}
The estimate can  be derived by a suitable integration by parts in one space dimension. We write and compute
\begin{align*}
\int_0^{\infty} \psi^2 r  e^{-\frac{r^2+z^2}4}\, dr &= \frac12 \int_{0}^{\infty} \psi^2 \left(\frac{d}{dr} r^2\right) e^{-\frac{r^2+z^2}4}\, dr\\
&  = -\int_0^{\infty} \psi \partial_r \psi r^2 e^{-\frac{r^2+z^2}4}\, dr + \frac14 \int_0^{\infty} \psi^2 r^3  e^{-\frac{r^2+z^2}4}\, dr,
\end{align*}
from which we derive with the help of Young's inequality that
\[
\int_0^{\infty} \psi^2 r  e^{-\frac{r^2+z^2}4}\, dr  \lesssim \int_0^{\infty} \psi^2 r^3  e^{-\frac{r^2+z^2}4}\, dr  + \int_0^{\infty} (\partial_r \psi)^2 r^3  e^{-\frac{r^2+z^2}4}\, dr.
\]
Integration in $z$ gives the desired result.
\end{proof}

\begin{namedlemma}[Gronwall-type inequalities]\label{gronwallversion}
          Let a function $F\colon \left[0,\infty\right)\to \left[0,\infty\right)$ and two constants $0\le \lambda < \mu $ be given.
          \begin{enumerate}[label=(\alph{enumi}),ref=(\alph{enumi})]
  \item \label{gronwallB}
          Assume that there exists a constant $C>0$ such that $F$ satisfies
               \begin{align}
                  \frac{d}{dt} F +(\lambda- Ce^{-t})  F\le C e^{-\mu t}\quad \mbox{for all }t\ge 0.
               \end{align}
          Then it holds that
               \begin{align}
                    F(t)\lesssim  e^{-\lambda t} \quad \text{ for all } t\geq 0.
               \end{align}
                 \item \label{gronwallC}
          Assume that there exists a constant $C>0$ such that $F$ satisfies
               \begin{align}
                  \frac{d}{dt} F +(\lambda- Ce^{-t})  F\le C e^{-\lambda t} \quad \mbox{for all }t\ge 0.
               \end{align}
          Then it holds that
               \begin{align}
                    F(t)\lesssim  (1+t) e^{-\lambda t} \quad \text{ for all } t\geq 0.
               \end{align}
          \end{enumerate}
     \end{namedlemma}

     \begin{proof}
     We can prove both statements simultaneously, allowing for $\mu=\lambda$.
          
          We define $\phi(s) = \lambda -Ce^{-s}$ and compute
               \begin{align}
                    \frac{d}{dt}\left(\exp\left(\int_0^t\phi(s)ds\right)F(t)\right) & \leq C\exp\left(\int_0^t\phi(s)ds-\mu t\right) \\
                    & = C\exp\left((\lambda-\mu)t - C(1-e^{-t})\right) \lesssim e^{(\lambda-\mu)t}.
               \end{align}
 Integration in time gives the respective results.
     \end{proof}

\setcounter{section}{1}
\setcounter{theorem}{0}

\section*{Acknowledgment}

	This work is funded by the Deutsche Forschungsgemeinschaft (DFG, German Research Foundation) under Germany's Excellence Strategy EXC 2044--390685587, Mathematics M\"unster: Dynamics Geometry Structure.

\bibliography{mybib}
\bibliographystyle{abbrv}
\end{document}